\documentclass{article}

\usepackage{amsmath,amssymb,amsthm}
\newtheoremstyle{slplain}
  {\topsep}
  {\topsep}
  {\slshape}
  {0pt}
  {\bfseries}
  {.}
  {0.5em}
  {}

\theoremstyle{slplain}
  \newtheorem{THM}{Theorem}[section]

\theoremstyle{definition}

\usepackage{color}
\usepackage{hyperref}
\hypersetup{colorlinks=true}

\newcommand{\spec}{\mathrm{spec}}
\newcommand{\Fraisse}{Fra\"\i ss\'e}
\newcommand{\Emb}{\mathrm{Emb}}

\newcommand{\calA}{\mathcal{A}}
\newcommand{\calB}{\mathcal{B}}
\newcommand{\calC}{\mathcal{C}}

\newcommand{\calH}{\mathcal{H}}

\newcommand{\calS}{\mathcal{S}}
\newcommand{\calU}{\mathcal{U}}

\newcommand{\NN}{\mathbb{N}}
\newcommand{\ZZ}{\mathbb{Z}}
\newcommand{\QQ}{\mathbb{Q}}

\newcommand{\KK}{\mathbf{K}}

\newcommand{\Boxed}[1]{\hbox{$#1$}}

\newcommand{\union}{\cup}
\newcommand{\UNION}{\bigcup}
\newcommand{\0}{\emptyset}
\renewcommand{\ge}{\geqslant}
\renewcommand{\le}{\leqslant}
\renewcommand{\phi}{\varphi}

\title{A short note on the characterization of\\countable chains with finite big Ramsey spectra}
\author{%
  Keegan Dasilva Barbosa\thanks{%
  Fields Institute for Research in Mathematical Science,
  222 College St, Toronto, ON M5T 3J1, Canada,
  keegan.dasilvabarbosa@mail.utoronto.ca},
  Dragan Ma\v sulovi\'c\thanks{%
  Department of Mathematics and informatics,
  Faculty of Sciences, University of Novi Sad,
  Trg Dositeja Obradovi\'ca 3, 21000 Novi Sad, Serbia,
  masul@dmi.uns.ac.rs} and
  Rajko Nenadov\thanks{%
  School of Computer Science, University of Auckland, New Zealand,
  rajko.nenadov@auckland.ac.nz}
}

\begin{document}
\maketitle

\begin{abstract}
  In this short note we confirm the deep structural correspondence between the complexity of a countable 
  scattered chain ($=$ strict linear order) and its big Ramsey combinatorics: we show that a countable scattered chain
  has finite big Ramsey degrees if and only if it is of finite Hausdorff rank.
  This also provides a complete characterization of countable chains whose big Ramsey spectra are finite.
  We expand the notion of big Ramsey spectrum to monomorphic structures
  and give a sufficient condition for a monomorphic countable structure to have finite big Ramsey spectrum.
\end{abstract}

\section{Introduction}

The Infinite Ramsey Theorem, usually abbreviated using the Erdo\H os-Rado partition arrow as
$\omega \longrightarrow (\omega)_2^n$ $(n \in \NN)$, is historically the first statement about 
Ramsey combinatorics of infinite structures. The importance of this statement was immediately recognized
and already in the early 1930's we have the first results of Sierpi\'nski pertaining to alephs,
which initiated the entire research field of combinatorial set theory. The generalization
of Ramsey's ideas in the orthogonal direction by Erd\H os, Rado, Ne\v set\v ril, R\"odl and
many more gave us in the early 1970's the structural Ramsey theory. 
The investigation of Ramsey combinatorics of
countable structures, and in particular, big Ramsey degrees was suggested by Kechris, Pestov and Todor\v cevi\'c
in~\cite{KPT}. The so-called KPT-correspondence describes a fascinating relationship between phenomena
in structural Ramsey theory and in topological dynamics of automorphism groups of \Fraisse\ limits.

In~\cite{masul-sobot} we presented an analysis of big Ramsey combinatorics of
countable structures that are as far away from \Fraisse\ limits as possible -- countable ordinals.
We would like to point out a related research in the topological setting~\cite{Gheysens},
as well as a recent paper where the big Ramsey degrees for countable ordinals have actually been computed~\cite{boyland-et-al}.
Interestingly, big Ramsey combinatorics of nonscattered chains is easy, as was
demonstrated by Galvin, Laver and Devlin \cite{galvin1,galvin2,devlin}. The situation with
scattered countable chains is more intriguing. It was shown in \cite{masul-fbrd-chains} that countable scattered
chains of infinite Hausdorff rank do not have finite big Ramsey spectra, and that \emph{some}
countable scattered chains of finite Hausdorff rank have finite big Ramsey spectra. 

In this short note we confirm the deep structural correspondence between the complexity of a countable 
scattered chain and its big Ramsey combinatorics: we show that a countable scattered chain
has finite big Ramsey degrees if and only if it is of finite Hausdorff rank.
This also presents a solution to the open problem in \cite{masul-fbrd-chains} and
provides a complete characterization of countable chains whose big Ramsey spectra are finite.

Big Ramsey spectra naturally generalize to the class of monomorphic structures introduced by \Fraisse~\cite{Fraisse-ThRel}.
An infinite relational structure $\calA$ is \emph{monomorphic}~\cite{Fraisse-ThRel} if, for each $n \in \NN$,
all the $n$-element substructures of $\calA$ are isomorphic.
The intimate relationship between chains and monomorphic structures will be presented in Section~\ref{fbrd-finale.sec.mnmf}.
A \emph{big Ramsey spectrum} of a monomorphic structure $\calS$, denoted by $\spec(\calS)$, is a sequence
$(T_1, T_2, T_3, \ldots)$ where $T_n \in \NN \cup \{\infty\}$ is the big Ramsey degree of the (only) $n$-element substructure of $\calS$.
We say that $\spec(\calS)$ is \emph{finite} if $T_n \in \NN$ for all $n \in \NN$.
The purpose of this short note is to give the complete characterization of countable chains with finite big Ramsey spectra,
and provide a sufficient condition for a monomorphic countable structure to have finite big Ramsey spectrum.
We conclude the paper with an open problem.

\section{Preliminaries}

Let $L$ be a relational language. For $L$-structures $\calA$ and $\calB$ let
$\Emb(\calA, \calB)$ denote the set of all the embeddings $\calA \hookrightarrow \calB$.
For $L$-structures $\calA$, $\calB$, $\calC$ and positive integers $k, t \in \NN$
we write $\calC \longrightarrow (\calB)^{\calA}_{k, t}$ to denote that for every $k$-coloring $\chi : \Emb(\calA, \calC) \to k$
there is an embedding $w \in \Emb(\calB, \calC)$ such that $|\chi(w \circ \Emb(\calA, \calB))| \le t$.
We say that $\calA$ has \emph{finite big Ramsey degree in $\calC$}
if there exists a positive integer $t$ such that for each $k \in \NN$ we have that
$\calC \longrightarrow (\calC)^{\calA}_{k, t}$.
The least such $t$ is then denoted by $T(\calA, \calC)$. If such a $t$ does not exist
we say that $\calA$ \emph{does not have finite big Ramsey degree in $\calC$} and write
$T(\calA, \calC) = \infty$. Finally, we say that an infinite $L$-structure $\calC$ \emph{has finite big Ramsey degrees}
if $T(\calA, \calC) < \infty$ for every finite substructure $\calA$ of~$\calC$.

For any monomorphic structure $\calS$ there is, up to isomorphism, only one $n$-element substructure,
so it is convenient to consider the \emph{big Ramsey spectrum of $\calS$}:
$$
  \spec(\calS) = (T(\calS_1, \calS), T(\calS_2, \calS), T(\calS_3, \calS), \ldots) \in (\NN \union \{\infty\})^\NN,
$$
where $\calS_n$ is the only $n$-element substructure of $\calS$. We then say that $\calS$
\emph{has finite big Ramsey spectrum} if $T(\calS_n, \calS) < \infty$ for all $n \ge 1$,
that is, if $\spec(\calS) \in \NN^\NN$.

For a linear order $<$ on $A$, its \emph{reverse} is the linear order $<^*$
defined so that $x \mathrel{<^*} y$ if and only if $y < x$.
When the ordering relation is not required explicitly or is obvious from the context,
we shall identify the chain with its underlying set, and the reverse of $A$ will be denoted by~$A^*$.

\section{Countable chains}

A chain $A$ is \emph{scattered} if $\QQ \not\hookrightarrow A$; otherwise it is \emph{non-scattered}.
In 1908 Hausdorff published a structural characterization of scattered chains~\cite{hausdorff-scat}, which was
rediscovered by Erd\H os and Hajnal in their~1962 paper~\cite{erdos-hajnal-scat}.
Define a sequence $\calH_\alpha$ of sets of chains indexed by ordinals as follows:
\begin{itemize}
\item
  $\calH_0 = \{0, 1\}$ -- the empty chain $\0$ and the 1-element chain 1;
\item
  for an ordinal $\alpha > 0$ let
  $
    \calH_\alpha = \{\sum_{i \in \ZZ} S_i : S_i \in \UNION_{\beta < \alpha} \calH_\beta \text{ for all } i \in \ZZ \}
  $.
\end{itemize}
Hausdorff then shows in~\cite{hausdorff-scat} that
  for each ordinal $\alpha$ the elements of $\calH_\alpha$ are countable scattered chains; and
  for every countable scattered chain $S$ there is an ordinal $\alpha$ such that $S$ is order-isomorphic
  to some chain in $\calH_\alpha$.
The least ordinal $\alpha$ such that $\calH_\alpha$ contains a chain order-isomorphic to a
countable scattered chain $S$ is referred to as the \emph{Hausdorff rank of $S$} and denoted by $r_H(S)$.
A countable scattered chain $S$ has \emph{finite Hausdorff rank} if $r_H(S) < \omega$;
otherwise it has \emph{infinite Hausdorff rank}.

\begin{THM}\label{fbrd-finale.thm.MAIN}
  Let $C$ be a countable chain. Then $\spec(C)$ is finite if and only if $C$ is non-scattered, or
  $C$ is a scattered chain of finite Hausdorff rank.
\end{THM}
\begin{proof}
  Big Ramsey spectra of all non-scattered countable chains are finite by results of Galvin, Laver and Devlin
  \cite{galvin1,galvin2,devlin}. Moreover, it was shown in \cite{masul-fbrd-chains} that countable scattered
  chains of infinite Hausdorff rank do not have finite big Ramsey spectra, and that
  countable scattered chains of finite Hausdorff rank and \emph{with bounded finite sums} 
  (see \cite{masul-fbrd-chains} for details) have finite big Ramsey spectra.
  We are now going to show that, up to biembeddability, there are only finitely
  many algebraically indecomposable types of rank $\le \alpha$ for any finite ordinal $\alpha$,
  whence follows that every countable scattered chain of finite Hausdorff rank is biembeddable with
  a countable scattered chain of finite Hausdorff rank and with bounded finite sums.
  This completes the characterization.

  Recall that Laver \cite{laver-decomposition} built all algebraically indecomposable
  scattered orders of finite rank up to biembeddability as follows:
  start with $\calA_0 = \{1\}$; at a successor step let $\calA_{\alpha+1} = \calA_\alpha \cup \{\phi : \phi$
  is an $\omega$- or $\omega^*$-unbounded sum of types from $\calA_\alpha \}$; and
  at a limit step put $\calA_\alpha = \bigcup_{\beta < \alpha} \calA_\beta$.
  
  Suppose by way of induction that, up to biembeddability, there are only finitely many types of rank $\le \alpha$
  for some finite $\alpha$. Given an $\omega$-unbounded (resp. $\omega^*$-unbounded)  sum of the form $\sum_{i < \omega} \Phi_i$,
  $\Phi_i \in \calA_\alpha$, let $U=\{\Phi_i : i < \omega\}$. By our hypothesis, $U$ is finite up to biembeddability.
  Let $\Psi_1, \ldots, \Psi_k$ be an enumeration of representatives from $U$.
  It can easily be verified by recursively building embeddings that $\sum_{i < \omega} (\Psi_1+\ldots + \Psi_k)$
  is biembeddable with $\sum_{i < \omega} \Phi_i$. Consequently, if there are $T_\alpha$ many
  representatives of elements of $\calA_\alpha$ up to biembeddability, then $T_{\alpha+1} \le 2 (2^{T_\alpha} - 1)$
  (first we choose $\omega$ or $\omega^*$, then we choose a nonempty subset of the $T_\alpha$ representatives to iterate and sum over). 
\end{proof}

\section{Monomorphic structures}
\label{fbrd-finale.sec.mnmf}

As demonstrated by \Fraisse\ in~\cite{Fraisse-ThRel}, and then generalized by Pouzet in~\cite{pouzet-mnmf=chainable},
monomorphic structures are closely related to chains.
Let $L = \{ R_i : i \in I \}$ and $M = \{ S_j : j \in J \}$ be relational languages. An $M$-structure
$\calA = (A, S_j^{\calA})_{j \in J}$ is a \emph{reduct} of an $L$-structure $\calA^*  = (A, R_i^{\calA^*})_{i \in I}$
if there exists a set $\Phi = \{ \phi_j : j \in J \}$ of $L$-formulas such that
for each $j \in J$ (where $\overline a$ denotes a tuple of elements of the appropriate length):
$$
  \calA \models S_j[\overline a] \text{ if and only if } \calA^* \models \phi_j[\overline a].
$$
We then say that $\calA$ is \emph{defined in $\calA^*$ by $\Phi$}, and that it is
\emph{quantifier-free definable in $\calA^*$} if there is a set of quantifier-free formulas $\Phi$
such that $\calA$ is defined in $\calA^*$ by $\Phi$.

A relational structure $\calA = (A, L^\calA)$ is \emph{chainable}~\cite{Fraisse-ThRel}
if there exists a linear order $<$ on $A$ such that $\calA$ is quantifier-free definable in~$(A, \Boxed<)$.
We then say that the linear order~$<$ \emph{chains}~$\calA$.
The following theorem was  proved by
\Fraisse\ for finite relational languages~\cite{Fraisse-ThRel} and for arbitrary relational languages by Pouzet~\cite{pouzet-mnmf=chainable}.

\begin{THM} \cite{Fraisse-ThRel,pouzet-mnmf=chainable}
  An infinite relational structure is monomorphic if and only if it is chainable.
\end{THM}

As a consequence of the above characterization of monomorphic structures, we have the following:

\begin{THM}
  Let $\calS$ be a countable monomorphic structure. If $\calS$ is chainable by a non-scattered chain, or by a
  scattered chain of finite Hausdorff rank, then $\spec(\calS)$ is finite.
\end{THM}
\begin{proof}
  Let $L$ be a relational language. Recall that a class $\KK$ of $L$-structures is called hereditary if the following holds:
  if $\calA \in \KK$ and $\calB$ is an $L$-structure which embeds into $\calA$, then $\calB \in \KK$. 
  An $L$-structure $\calU$ is universal for $\KK$ if $\calU$ embeds every $\calA \in \KK$.

  Let us also recall \cite[Theorem~7.1]{masul-rdbas}.
  Let $L = \{ R_1, \ldots, R_n \}$ be a finite relational language, let
  $M = \{ S_j : j \in J \}$ be a relational language and let
  $\Phi = \{ \phi_j : j \in J \}$ be a set of quantifier-free $L$-formulas.
  Let $\KK^*$ be a hereditary class of at most countably infinite $L$-structures and let
  $\KK$ be the class of all the $M$-structures which are definable by $\Phi$ in~$\KK^*$.
  Let $\calS^* \in \KK^*$ be universal for $\KK^*$ and let $\calS \in \KK$ be the $M$-structure
  defined in $\calS^*$ by $\Phi$. Then \cite[Theorem~7.1]{masul-rdbas} shows that
  $\calS$ is universal for $\KK$, and
  if $\calS^*$ has finite big Ramsey degrees, then so does $\calS$.

  Let $\calS$ be a countable monomorphic structure which is chainable by a chain $C$ which is either non-scattered, or
  scattered of finite Hausdorff rank. By Theorem~\ref{fbrd-finale.thm.MAIN} we know that $\spec(C)$ is finite.
  As we have just seen, $\spec(\calS)$ is then finite by \cite[Theorem~7.1]{masul-rdbas}.
\end{proof}

\paragraph{Open problem.}
Let $\calS$ be a countable monomorphic structure which is chainable by a scattered chain of infinite Hausdorff rank.
Is it true that $\spec(\calS)$ then cannot be finite?

\section*{Declarations}

\paragraph{Competing interests.} No competing interests.

\paragraph{Authors' contributions.} All authors have contributed to the main result.
All authors have reviewed the paper.

\paragraph{Funding.}
The first author was supported by the \href{http://www.fields.utoronto.ca/}{Fields Institute for Research in Mathematical Sciences}.
The second author was supported by the Science Fund of the Republic of Serbia, Grant No.~7750027:
Set-theoretic, model-theoretic and Ramsey-theoretic phenomena in mathematical structures: similarity and diversity -- SMART.

\paragraph{Availability of data and materials.}
No datasets were used or generated in the course of the preparation of the paper.


\begin{thebibliography}{99}
\bibitem{boyland-et-al}
  J.\ Boyland, W.\ Gasarch, N.\ Hurtig, R.\ Rust.
  Big Ramsey Degrees of Countable Ordinals.
  (preprint available as arXiv:2305.07192)

\bibitem{devlin}
  D. Devlin.
  Some partition theorems and ultrafilters on $\omega$.
  Ph.D. Thesis, Dartmouth College, 1979.

\bibitem{erdos-hajnal-scat}
  P. Erd\H os, A. Hajnal.
  On a classification of denumerable order types and an application to the partition calculus.
  Fundamenta Mathematicae 51 (1962), 117--129.

\bibitem{Fraisse-ThRel}
  R.\ \Fraisse.
  Theory of relations, Revised edition, With an appendix by Norbert Sauer.
  Studies in Logic and the Foundations of Mathematics, 145.
  North-Holland, Amsterdam, 2000

\bibitem{galvin1}
  F. Galvin.
  Partition theorems for the real line.
  Notices Amer. Math. Soc. 15 (1968), 660.

\bibitem{galvin2}
  F. Galvin.
  Errata to ``Partition theorems for the real line''.
  Notices Amer. Math. Soc. 16 (1969), 1095.

\bibitem{Gheysens}
  M.\ Gheysens.
  Dynamics and structure of groups of homeomorphisms of scattered spaces.
  Enseign.\ Math.\ 67 (2021), no.\ 1/2, 161--186

\bibitem{hausdorff-scat}
  F. Hausdorff.
  Grundz\"uge einer Theorie der Geordnete Mengen.
  Mathematische Annalen 65 (1908), 435--505.

\bibitem{KPT}
  A.\ S.\ Kechris, V.\ G.\ Pestov, S.\ Todor\v cevi\'c.
  Fra\"\i ss\'e limits, Ramsey theory and topological dynamics of automorphism groups.
  GAFA Geometric and Functional Analysis, 15 (2005) 106--189.

\bibitem{laver-decomposition}
  R.\ Laver.
  An Order Type Decomposition Theorem.
  Annals of Mathematics, Second Series, 98(1973), 96--119.

\bibitem{masul-rdbas}
  D. Ma\v sulovi\'c.
  Ramsey degrees: big v. small.
  European Journal of Combinatorics, 95 (2021), Article 103323.

\bibitem{masul-fbrd-chains}
   D.\ Ma\v sulovi\'c.
   Big Ramsey spectra of countable chains.
   Order 40 (2023), 237--256.

\bibitem{masul-sobot}
  D. Ma\v sulovi\'c, B. \v Sobot.
  Countable ordinals and big Ramsey degrees.
  Combinatorica 41 (2021), 425--446.

\bibitem{pouzet-mnmf=chainable}
   M. Pouzet.
   Application de la notion de relation presque-encha\^\i nable au d\'enombrement des restrictions finies d'une relation.
   Z. Math. Logik Grundlag. Math., 27 (1981) 289--332.

\end{thebibliography}
\end{document}